\documentclass[11pt]{amsart} 

\title[Ideals in O.S. Projective Tensor Product]{Ideals in Operator
  Space Projective Tensor Product of $C^*$-algebras}

\usepackage[all]{xy}
\usepackage{amsmath,amsfonts,amssymb,amsthm,a4,hyperref}
\newtheorem{theorem}{\sc Theorem}[section]
\newtheorem{corollary}[theorem]{\sc Corollary}
\newtheorem{prop}[theorem]{\sc Proposition}

\newtheorem{remark}[theorem]{\sc Remark}
\newtheorem{lemma}[theorem]{\sc Lemma}

\newcommand{\oop}{\widehat\otimes}
\newcommand{\ra}{\rightarrow}
\def\NN{\mathbb{N}}

\begin{document}

\author[R. Jain]{Ranjana Jain} \address{Department of Mathematics\\
  Lady Shri Ram College for Women\\ New Delhi-110024, India.}
\email{ranjanaj\_81@rediffmail.com}

\author[A. Kumar]{Ajay Kumar}
\address{Department of Mathematics\\ University of Delhi\\ Delhi-110007\\
  India.}  \email{akumar@maths.du.ac.in}



\begin{abstract} 
  For $C^*$-algebras $A$ and $B$, we prove the slice map conjecture
  for ideals in the operator space projective tensor product $A
  \widehat\otimes B$. As an application, a characterization of prime
  ideals in the Banach $\ast$-algebra $A\widehat\otimes B$ is
  obtained. Further, we study the primitive ideals, modular ideals and
  the maximal modular ideals of $A\widehat\otimes B$. It is also shown
  that the Banach $\ast$-algebra $A\widehat\otimes B$ possesses Wiener
  property; and that, for a subhomogeneous $C^*$-algebra $A$, $A
  \widehat\otimes B$ is symmetric.
\end{abstract}

\maketitle

\section{ Introduction}
A systematic study of tensor products of subspaces and subalgebras of
$C^*$-algebras was initiated by Blecher and Paulsen \cite{blepau}, and
Effros and Ruan \cite{effrua1,effrua2}. Analogous constructions to
those of Banach spaces; for example, quotients, duals and tensor
products were defined and studied. For a Hilbert space $H$, let
$B(H)$ denote the bounded operators on $H$. An operator space $X$ on
$H$ is just a closed subspace of $B(H)$. If $E$ and $F$ are operator
spaces, then the operator space projective tensor product, denoted by
$E\widehat{\otimes}F$, is the completion of the algebraic tensor
product $E\otimes F$ under the norm
\[
\|u\|_{_{\wedge}} = \inf\{\|\alpha\|\|v\|\|w\| \|\beta\| : u = \alpha
(v \otimes w) \beta \},\] where the infimum runs over arbitrary
decompositions with $ v\in M_p(E),\,  w\in M_q(F)$, $\alpha \in M_{1,\, pq}$, 
$\beta\in M_{pq,\, 1}$ with $ p, q\in \NN $ arbitrary; $M_{k,\, l}$ being
the space of $k\times l$ matrices over $\mathbb C$. If $E$ and $F$ are
$C^*$-algebras, then $E\widehat\otimes F$ admits a Banach algebra with  canonical isometric involution \cite{kumar}. The main objective of this paper is to study the closed $*$-ideals of this Banach $*$-algebra.
 
In Section 2, we study the slice map problem for ideals of
$A\widehat\otimes B$. Tomiyama \cite{tomiyama1} studied the slice maps
on the tensor product of $C^*$-algebras with respect to the
`min'-norm. Later, Wassermann \cite{wass} discussed the slice map
problem in greater detail, which was then studied and used in
different contexts - see, for instance, \cite{archbold, wass2}. It is
interesting to know that the slice map property is not true for the
`min' norm for all $C^*$-algebras. In fact, for the `min' norm the
slice map problem for ideals is equivalent to the problem of whether every tensor
product $A\otimes_{\min} B$ has Property F of Tomiyama \cite[Remark
24]{wass}. In 1991, Smith \cite{smith} studied the slice map property
for the Haagerup norm and proved that the slice map conjecture is true
for all subspaces of $B(H)$. We give an affirmative answer to the
slice map conjecture for ideals with respect to the operator space
projective tensor norm.

The ideal structure for the Haagerup tensor product and the `min' norm
has been studied extensively in \cite{ass}, \cite{arch} and
\cite{takesaki}. In \cite{kumar} and \cite{ranj2}, the authors
investigated some properties of the closed ideals of the projective tensor
product $A\widehat\otimes B$, for example, sum of the product ideals, minimal and the maximal ideals. In Section 3, we 
discuss a characterization of prime ideals, primitive ideals, and maximal
modular ideals of the Banach $\ast$-algebra $A\widehat\otimes
B$. Finally, in Section 4, certain $*$-algebraic properties of
$A\widehat\otimes B$, namely, Wiener property and symmetry are
studied. Throughout the paper, $A$ and $B$ will denote $C^*$-algebras
unless otherwise specified.

Recall that the Haagerup norm of an element $u$ in the algebraic
tensor product $A\otimes B $ of two $C^*$-algebras $A$ and $B$ is
defined by
$$
\|u\|_h = \inf \{\|\Sigma_i\, a_ia^*_i\|^{1/2}\ \|\Sigma_i\,
b^*_ib_i\|^{1/2}:\, u=\Sigma^{n}_{i=1}a_i\otimes b_i\}.$$
The Haagerup tensor product
$A\otimes_h B$ is defined to be the completion of $A\otimes B$ in the
norm $\|\cdot\|_h$. Also, the Banach space projective norm of $u \in
A\otimes B$ is given by
$$ \|u\|_\gamma = \inf \{\Sigma_i\, \|a_i\| \|b_i\| \,: \, 
u = \Sigma_{i=1}^{n}\ a_i \otimes b_i\}.
$$ 
The norms $\|\cdot\|_h,\|\cdot\|_{_{\wedge}}$ and $\|\cdot\|_\gamma$
on the tensor product $A\otimes B$ of two $C^*$-algebras $A$ and $B$
satisfy
$$
\|\cdot\|_h\, \leq\, \|\cdot\|_{_{\wedge}}\, \leq\, \|\cdot\|_\gamma.
$$
Necessary and sufficient conditions on $A$ and $B$ for the equivalence of these norms can be seen in \cite{kumar1}.


\section{Slice Map Property for Ideals}

\noindent For each $\phi \in A^*$, define a linear map $R_{\phi}:A
\otimes B \rightarrow B$ by
$$
R_{\phi}(\Sigma_{i=1}^{n}\, a_i \otimes b_i)=\Sigma_{i=1}^{n}\, \phi (a_i)
b_i.
$$
Then, it can be easily seen that $R_\phi$ is well defined. Also, it is
continuous with respect to the `min'-norm \cite{wass} and hence for the larger operator space projective tensor norm
with $\|R_{\phi} \| \leq \|\phi\|$; so, it can be extended to
$A\widehat\otimes B$ as a bounded linear map and is known as the {\em
  right slice map} associated to $\phi$. Similarly, one can define the
left slice map $L_{\psi}$ for each $\psi \in B^*$. For a closed ideal
$J$ of $B$, $A\widehat\otimes J$ is a closed ideal of
$A\widehat\otimes B$ \cite{kumar} and clearly $R_{\phi}(x) \in J$ for
all $x\in A\widehat\otimes J$.  We prove the converse of this
statement which is known as the slice map problem for ideals.

\begin{lemma}\label{total} 
  The set \{$R_\phi : \phi \in A^*$\} is total on $A\widehat\otimes
  B$, that is, if $x \in A\widehat\otimes B$ and $R_\phi (x)=0$ for
  all $\phi \in A^*$, then $x=0$.
\end{lemma}

\begin{proof}
  For $\phi \in A^*$ and $\psi \in B^*$, consider $\phi \otimes \psi:
  A\otimes B \rightarrow \mathbb{C}$ given by
$$
(\phi \otimes \psi)( \Sigma_i\, a_i \otimes b_i) = \Sigma_i\,
\phi(a_i) \psi(b_i).
$$
Note that, by the definition of  the Banach space injective norm $\lambda$  \cite[page 188]{takesaki}, we have $|\Sigma_i\,\phi(a_i) \psi(b_i)| \leq \|\phi\| \|\psi\| \| \Sigma_i a_i \otimes b_i\|_\lambda$. Thus $\phi \otimes \psi$ is continuous with respect to larger norms, in particular, `min'-norm and `$\wedge$'-norm; so, $\phi \otimes \psi$ can be
extended to continuous linear functionals on $A\otimes_{\min} B$ and
$A\widehat\otimes B$. Let us denote its extensions by $\phi
\otimes_{\min} \psi$ and $\phi \widehat\otimes \psi$
respectively. We claim that the set $\{\phi \widehat\otimes \psi : \phi \in A^*, \, \psi \in B^* \}$ is total
on $A \widehat\otimes B$. For this, consider  an element $x \in A\widehat\otimes B$ such that
 $$ (\phi \widehat\otimes \psi )(x)=0, \, \forall \phi \in A^*, \, \psi \in B^*.$$
Observe that for the canonical map $i: A\widehat\otimes B \rightarrow A\otimes_{\min}B$, the maps $\phi \widehat\otimes \psi$ and $(\phi \otimes_{\min} \psi)\circ i $ both are continuous on $A\widehat\otimes B$ and agree on $A\otimes B$, giving $(\phi \otimes_{\min} \psi)(i(x))=0$ for all  $\phi \in A^*, \, \psi \in B^*$. Now, for faithful representations $\{\pi_A,H\}$ and $\{\pi_B,K\}$ of $A$ and $B$ respectively, for $\xi_i \in H,\, \eta_i\in K$, $i=1,2$; $ \phi:= \langle\pi_{A}(\cdot) \xi_1,\xi_2 \rangle\,\in A^* ,\, \psi:= \langle \pi_{B}(\cdot) \eta_1,\eta_2\rangle\, \in B^*$; so
$$0=(\phi \otimes_{\min} \psi)(i(x))= \langle(\pi_A \otimes \pi_B)(i(x))\xi_1\otimes \eta_1,\xi_2\otimes \eta_2 \rangle.$$
This holds for all $\xi_i \in H,\,\eta_i\in K;\,i=1,2$ giving $(\pi_A \otimes \pi_B) (i(x))=0$. Using the facts that $\pi_A \otimes \pi_B$ is faithful \cite[Theorem IV.4.9]{takesaki}, and that $i$ is injective \cite[Corollary 1]{jain}
we obtain the claim. Finally, the relation
$$  \langle x, \phi \widehat\otimes \psi \rangle = \langle R_{\phi}(x),
\psi \rangle = \langle L_\psi (x), \phi \rangle,\, \forall x \in
A\widehat\otimes B,$$
gives the required result.
\end{proof}
Recall that, for Banach spaces $X$ and $Y$, a mapping $\theta:X\rightarrow Y$ is said to be a {\it quotient map} if it maps the open unit ball of $X$ onto that of $Y$ \cite{effrua1}. Clearly, a quotient map is surjective, and  for Banach space $X$ and a closed subspace $Y$ of $X$, the canonical quotient map $\pi:X \ra X/Y$ is  a quotient map in the above sense. Like in the case of Haagerup tensor product \cite{ass}, the operator space projective tensor product of quotient maps behaves nicely. Although straight forward, we include a proof of the following for the sake of convenience:
\begin{lemma}\label{ker}
 Let $I$ and $J$ be closed ideals of the $C^*$-algebras $A$ and $B$, and $\pi:A \rightarrow A/I$ and $\rho:B \rightarrow B/J$ be the quotient maps. Then,
\begin{enumerate}
 \item $\pi \oop \rho: A \oop B \ra (A/I) \oop (B/J)$ is a quotient map with $$\ker(\pi \oop \rho) = A \oop J + I\oop B.$$
 \item for a closed ideal $K$ of $A\oop B$ containing $\ker(\pi \oop \rho)$, $(\pi \oop \rho) (K)$ is a closed ideal of $(A/I) \oop (B/J)$ with $$ (\pi \oop \rho)^{-1}((\pi \oop \rho)(K)) = K.$$
\end{enumerate}
\end{lemma}

\begin{proof}
  (1) This follows directly from \cite[Proposition 3.5]{ranj2}. \\
(2) Consider an element $(\pi \oop \rho )(x) \in (A/I)\oop (B/J)$ such that $(\pi \oop \rho )(x) \in \text{cl} ((\pi\oop \rho)(K))$, where $x \in A\widehat\otimes B$. Given an arbitrary $\epsilon >0 $, there exists $k \in K$ such that $$\|(\pi \oop \rho)(k-x)\|_{(A/I) \oop (B/J)} < \epsilon.$$ Using part (1) above, there is an isomorphism between $(A \oop B)/ Z$ and $(A/I)\oop\\ (B/J)$, where $Z = \ker(\pi \oop \rho)$. Therefore, $$\|(k-x) + Z\|_{(A \oop B)/Z} < c\epsilon,$$ for some constant $c$. So, there exists some $z\in Z \subseteq K$ with $\|(k + z) -x\|_{(A \oop B)/Z} \leq c\epsilon$. Since $K$ is closed and $k+z \in K$, we must have $x\in K$, which proves the claim. Finally, the equation in the statement is a routine verification.
\end{proof}

 We are now prepared to present a proof of the slice map problem for ideals.
\begin{theorem}\label{slice}
  Let $J$ be a closed ideal of $B$. Then
$$
A\widehat\otimes J= \{ x \in A\widehat\otimes B : R_\phi (x) \in J \
\text{for all}\   \phi \in A^*\}.$$
\end{theorem}

\begin{proof}
 Consider an element $ x \in A\widehat\otimes B$ such that $R_\phi (x) \in J $ for all $\phi \in A^*$. From Lemma \ref{ker}, corresponding to the quotient map $\pi: B
  \rightarrow B/J$, we have another quotient map $ i \widehat \otimes
  \pi: A\widehat\otimes B \rightarrow A\widehat\otimes (B/J)$ with $\ker(i\widehat\otimes \pi) = A\widehat\otimes J$, where  `$i$' is the identity map on $A$. Also observe that, by continuity and agreement on $A\otimes B$, 
$$ \pi \circ R_\phi = r_{\phi} \circ (i \widehat\otimes \pi),$$
where $r_\phi:A\widehat\otimes (B/J) \rightarrow B/J$ is the right
slice map. Using the fact that $R_\phi (x) \in J$ for all $\phi \in
A^*$, we obtain $r_\phi ( i\widehat\otimes \pi (x)) = 0$ for all $\phi
\in A^*$. Thus, by Lemma \ref{total}, $i\widehat\otimes \pi (x) = 0$;
so that $x\in \ker(i\widehat\otimes \pi) = A\widehat\otimes J$. The other containment is easy.
\end{proof}

We next give an application of Theorem \ref{slice} which
will be used later to characterize the prime ideals. For the Haagerup
norm such a result was proved for subspaces of $B(H)$ in
\cite[Corollary 4.6]{smith}.

\begin{prop}\label{intersection}
  Let $A_1$, $A_2$ and $B_1, B_2$ be closed ideals of $A$ and $B$,
  respectively. Then,
$$ 
(A_1 \widehat\otimes B_1) \cap (A_2 \widehat\otimes B_2) = (A_1 \cap
A_2) \widehat\otimes (B_1 \cap B_2).$$
\end{prop}

\begin{proof}
  Since $A_i \widehat\otimes B_i,\, i=1,2$ are closed ideals of $A\widehat\otimes B$ \cite{kumar}, it is easy to see that 
$$
(A_1 \cap A_2) \widehat\otimes (B_1 \cap B_2) \subseteq (A_1
\widehat\otimes B_1) \cap (A_2 \widehat\otimes B_2).
$$
For the other containment, consider an element $v \in (A_1
\widehat\otimes B_1) \cap (A_2 \widehat\otimes B_2)$. Then, $R_\phi(v)
\in B_1 \cap B_2$ for all $\phi \in A^*$; so, by Theorem \ref{slice},
$v\in A \widehat\otimes (B_1 \cap B_2)$.  Next, consider any $\psi \in
(B_1 \cap B_2)^*$ and let $\tilde{\psi}$ be an extension on
$B^*$. Again, $L_{\tilde{\psi}}(v) \in (A_1 \cap A_2)$ and $L_\psi(v)
= L_{\tilde{\psi}}(v)$; so that $L_\psi(v) \in (A_1 \cap A_2)$. This
is true for every $\psi \in (B_1 \cap B_2)^*$; so, applying the slice
map property once again for the left slice map, we obtain $v \in (A_1
\cap A_2)\widehat\otimes (B_1 \cap B_2)$, which proves the claim.
\end{proof} 

Using the slice map property for the right and the left slice maps, and the technique of extending linear functionals as done in Proposition \ref{intersection}, we can easily deduce the following:

\begin{corollary}
For closed ideals $I$ and $J$ of $A$ and $B$ respectively, we have
$$ I \widehat\otimes J = \{ x \in A\widehat\otimes B : \,\, R_\phi(x) \in J ,\, L_\psi(x) \in I; \, \forall \, \phi \in A^*, \, \forall \,\psi \in B^* \}.$$
\end{corollary}


\section{\texorpdfstring{Ideal Structure for $A\widehat\otimes B$}{Ideal Structure for A oop B}}

This section deals with the structure of prime ideals, primitive
ideals and modular ideals of $A\widehat\otimes B$ which play an
important role in determining the structure of a Banach
$*$-algebra. In a Banach algebra a proper closed ideal $K$ is said to be 
{\it prime} if for any pair of closed ideals $I$ and $J$ satisfying
$IJ \subseteq K$, either $I \subseteq K$ or $J\subseteq K$. It is well known
that a proper  closed ideal $K$ of a $C^*$-algebra A is prime if and only if
for any pair of closed ideals $I$ and $J$ satisfying $I \cap J \subseteq
K$, either $I \subseteq K$ or $J\subseteq K$.  This property is also true
for $A\widehat\otimes B$ as can be explicitly observed from the
following result. The proof of the following result is largely
inspired by \cite{ass}.
\begin{theorem}
  A closed ideal $K$ in $A\widehat\otimes B$ is prime if and only if
  $K =A\widehat\otimes F + E\widehat\otimes B$ for some prime ideals
  $E$ and $F$ in $A$ and $B$ respectively.
\end{theorem} 

\begin{proof}
  Let $K$ be a closed prime ideal. We can choose closed ideals $E$ and
  $ F$ in $A$ and $B$ which are maximal with respect to the property
  $A\widehat\otimes F + E\widehat\otimes B \subseteq K$. Now consider
  the quotient maps $\pi:A \rightarrow A/E$ and $\rho:B\rightarrow
  B/F$. Since $\ker (\pi \otimes \rho) \subseteq K$, by Lemma \ref{ker}, $(\pi \otimes \rho)(K)$ is a closed ideal of $A/E
  \widehat\otimes B/F$. We claim that $(\pi \otimes \rho)(K) =0$; this
  would imply $K =A\widehat\otimes F + E\widehat\otimes B$. If possible, let the ideal $(\pi \otimes \rho)(K)$ be
  non-zero. Then, it must contain a non-zero elementary tensor, say,
  $\pi(a) \otimes \rho(b)$, where $a\otimes b \in K$ \cite[Proposition
  3.7]{ranj2}. Let $E_0$ and $F_0$ be the closed ideals generated by
  $a$ and $b$ respectively. Then, the product ideal $E_0
  \widehat\otimes F_0$ is contained in K. Now, consider the product ideals $M=
  A\widehat\otimes (F+ F_0)$ and $N= (E + E_0) \widehat\otimes
  B$. Using Proposition \ref{intersection} and \cite[Proposition
  3.6]{ranj2}, we have
$$ 
MN \subseteq M \cap N = E \widehat\otimes F + E\widehat\otimes F_0 +
E_0\widehat\otimes F + E_0 \widehat\otimes F_0 .
$$
It is clear that $MN \subseteq K$, so that either $M\subseteq K$ or
$N\subseteq K$. Using the maximality property of $E$ and $F$, we have
either $E_0 \subseteq E $ or $F_0 \subseteq F$. Thus, either $\pi (a) =0$
or $ \rho (b) = 0$ contradicting the fact that $(\pi \otimes
\rho)(a\otimes b) \neq 0$.

Next we prove that $E$ and $F$ are prime ideals. Note that $E$
and $F$ both are proper ideals, $K$ being proper. Let $I \cap J \subseteq
E$ for some closed ideals $I$ and $ J$ of $A$. Then, $(I \widehat\otimes B)(J \widehat\otimes B)
\subseteq (I \widehat\otimes B) \cap (J \widehat\otimes B) \subseteq K$;
so, either $I \widehat\otimes B \subseteq K$ or $J \widehat\otimes B
\subseteq K$. Without loss of generality, let $I \widehat\otimes
B\subseteq K$. Consider any $\phi \in E^{\perp} \subseteq A^*$ and $0 \neq
\psi \in F^{\perp}$. Then, $(\phi \otimes \psi)(K) = 0$ which further
gives $(\phi \otimes \psi)(I\widehat\otimes B) = 0$. Since this is
true for any $\phi \in E^\perp$, we must have $I \subseteq E$. Thus, $E$
is prime and by a similar argument $F$ is also prime. 

For the converse, let us assume that $K =A\widehat\otimes F +
E\widehat\otimes B$ for some prime ideals $E$ and $F$ in $A$ and $B$
respectively. Let $IJ \subseteq K$ for some closed ideals $I$ and $J$ of
$A\widehat\otimes B$. Define the closed ideals $M$ and $N$ as
$$ 
M = \text{cl} (I +K)\ \text{and} \ N = \text{cl} (J +K).
$$
Then $K \subseteq M, \, K \subseteq N$ and $MN \subseteq K$. We claim that
either $M= K$ or $ N = K$, which further implies that either $I
\subseteq K$ or $J \subseteq K$. Suppose, on the contrary, that both the
containments $K \subseteq M$ and $K \subseteq N$ are strict. We now claim
that $M$ contains a product ideal $M_1 \widehat\otimes N_1$ which is
not contained in $K$. As done previously, since $K \subsetneq M,\,
(\pi \otimes \rho) (M)$ is a non-zero closed ideal of $A/E
\widehat\otimes B/F$ with $(\pi \otimes \rho)^{-1} ((\pi \otimes \rho
)(M))= M$. So, $(\pi \otimes \rho) (M)$ contains a non-zero elementary
tensor say $\pi(a) \otimes \rho(b)$. Define $M_1$ and $N_1$ to be the
closed ideals generated by $a$ and $b$. Then $M_1\widehat\otimes N_1$
is contained in $M$ but not in $K$. Similarly, $N$ contains a product
ideal $M_2 \widehat\otimes N_2$ which is not contained in $K$. By
routine calculations, it is easily seen that
$$
M_1M_2 \widehat\otimes N_1N_2= \text{cl}((M_1 \widehat\otimes N_1)(M_2
\widehat\otimes N_2)) \subseteq \text{cl}(MN) \subseteq K,
$$
which further gives
$$ 
\pi (M_1M_2)\otimes \rho(N_1N_2) \subseteq (\pi \otimes
\rho)(M_1M_2\widehat\otimes N_1N_2)=\{0\}.
$$ 
So either $M_1M_2 \subseteq \ker \pi = E$ or $N_1N_2 \subseteq
\ker \rho = F$. Now, both $E$ and $F$ are prime, so at least one of
the following containments must hold:
$$
M_1\subseteq E,\,\, M_2\subseteq E,\,\,N_1\subseteq F,\,\,N_2\subseteq F.
$$ 
In all these cases, either $M_1\widehat\otimes N_1 $ or
$M_2\widehat\otimes N_2$ is contained in $K$, which is a
contradiction. Thus, $K$ is prime.
\end{proof}

A closed ideal $I$ of a Banach $\ast$-algebra $E$ is said to be {\it primitive}
if it is the kernel of an irreducible $\ast$-representation of $E$ on some Hilbert space. The following gives a characterization of the
primitive ideals of $A\widehat\otimes B$.

\begin{theorem}\label{prim}
For  $C^*$-algebras $A$ and $B$, we have the following:
\begin{enumerate}
\item If $E$ and $F$ are primitive ideals of $A$ and $B$ respectively,
  then $A\widehat\otimes F + E\widehat\otimes B$ is also a primitive
  ideal of $A\widehat\otimes B$.
\item If $K$ is a primitive ideal of $A\widehat\otimes B$, then $K=
  A\widehat\otimes F + E\widehat\otimes B$ for some prime ideals $E$
  and $F$ of $A$ and $B$, respectively.
\item If $A$ and $B$ are separable, then $K$ is primitive if and only
  if $K= A\widehat\otimes F + E\widehat\otimes B$ for some primitive
  ideals $E$ and $F$ of $A$ and $B$, respectively.
\end{enumerate}
\end{theorem}

\begin{proof}
  (1) Since $E$ and $F$ are primitive ideals, there exist irreducible
  $\ast$- representations $\pi_1: A \rightarrow B(H_1)$ and $\pi_2:B
  \rightarrow B(H_2)$ such that $E=\ker \pi_1$ and $F=\ker
  \pi_2$. Define $\pi: A\otimes B \rightarrow B(H_1 \otimes H_2)$ by
$$ 
\pi(a\otimes b) = \pi_1 (a)\otimes \pi_2(b). 
$$
Then, by the definition of min-norm \cite{takesaki}, $\pi$ is bounded with respect to the min-norm and hence the `$\wedge$' norm; so, $\pi$ can
be extended to $A\widehat\otimes B$ as a bounded
$\ast$-representation. We first claim that $\pi$ is irreducible,
equivalently, $\pi(A\widehat\otimes B)^{\prime} =
\mathbb{C}I$. 
Since $\pi(A \widehat\otimes B) \supset \pi_1(A) {\otimes} \pi_2(B)$, we have $\pi(A\widehat\otimes B)^{\prime} \subseteq
(\pi_1(A) \overline{\otimes} \pi_2(B))^{\prime}$, where $\overline{\otimes}$ denotes the weak closure. Further, $\pi_1$ and $\pi_2$ being irreducible, $\pi_1(A)$ and $\pi_2(B)$ are non-degenerate $\ast$-subalgebras of $B(H_1)$ and $B(H_2)$, respectively; so that, by Double Commutant Theorem, $\pi_1 (A)$ and $\pi_2(B)$ are weakly dense in $\pi_1(A)''$ and $\pi_2(B)''$. In particular,  $ \pi_1(A) \overline{\otimes} \pi_2(B) = \pi_1(A)''\, \overline{\otimes} \pi_2(B)''$; and, an appeal to Tomita's Commutation Theorem then yields   $ (\pi_1(A) \overline{\otimes} \pi_2(B))^{\prime}= \pi_1(A)^{\prime}
\,\overline{\otimes} \pi_2(B)^{\prime} \subseteq \mathbb{C}I$, which shows that $\pi$ is irreducible.
\vspace*{1mm}

Next we claim that $\ker \pi =A \widehat\otimes F + E \widehat\otimes
B =K (\text{say})$. Clearly, $A\widehat\otimes F$ and $E\widehat\otimes
B$ are both contained in $\ker \pi$; so that $K \subseteq \ker \pi$. For
the other containment, consider the quotient map $\theta:A
\widehat\otimes B \rightarrow A/E \widehat\otimes B/F$ with $\ker
\theta = K$. Since, $\ker\pi$ contains
$\ker \theta$, by Lemma \ref{ker}, $\theta(\ker \pi)$ is a closed ideal of $A/E
\widehat\otimes B/F$ with $\theta^{-1}(\theta(\ker \pi))=\ker \pi$. If
$\theta(\ker \pi) \neq 0$, then it must contain a non-zero elementary
tensor say $(a+E) \otimes (b+F)$ \cite[Proposition 3.7]{ranj2}. Now
$a\otimes b \in \ker \pi$ implies $\pi_1(a) \otimes \pi_2(b) = 0$,
which further implies that either $a\in E$ or $b \in F$, so that
$(a+E) \otimes (b+F) = 0$, which is a contradiction. Thus, $\ker \pi
\subseteq \ker \theta =K$.

\vspace*{2mm}

(2) Let $K= \ker \pi$ for some irreducible $\ast$-representation $\pi$
of $A\widehat\otimes B$ on $H$. By \cite[Lemma IV.4.1]{takesaki}, there
exist commuting $\ast$-representations $\pi_1:A \rightarrow B(H)$ and
$\pi_2:B\rightarrow B(H)$ such that 
$$
\pi(a\otimes b) =\pi_1(a) \pi_2(b),\, \forall\, a\in A, \, b \in B.
$$
Now, $\pi(A\otimes B)= \pi_1(A) \pi_2(B)$, so $ \pi(A\widehat\otimes
B) \subseteq \text{cl}(\pi_1(A) \pi_2(B))$. Thus, we obtain
$$
(\pi_1(A) \pi_2(B))^{'}=\text{cl}(\pi_1(A) \pi_2(B))^{'} \subseteq
\pi(A\widehat \otimes B)^{'} = \mathbb{C} I.
$$ 
Also, note that $\pi_1$ and $\pi_2$ are both factor representations as
for $P= \pi_1(A)^{''}$ and $ Q=\pi_2(B)^{''}$, we have
\begin{eqnarray*}
  P \cap P^{'} & = & \pi_1(A)^{''} \cap \pi_1(A)^{'}\\
  &=& (\pi_1(A)^{'} \cup \pi_1(A))^{'} \\
  &\subseteq & (\pi_2(B) \cup \pi_1(A))^{'} \, \qquad \qquad
  ( \text{as} \, \pi_1(A)\, \text{and} \, \pi_2(B)\, \text{commute})\\
  & \subseteq & \{\pi_1(A) \pi_2(B) \}^{'}\\
  &=& \mathbb{C} I.
\end{eqnarray*}
Now, let $E=\ker \pi_1$ and $F=\ker \pi_2$. Then $E$ and $F$, being
kernels of factor representations, are both prime ideals  
\cite[II.6.1.11]{blackadar}. Also, by the definition of $\pi$,
$A\widehat\otimes F + E\widehat\otimes B \subseteq K$. For the reverse
containment, consider $a\otimes b\in K$. Then, we have
$\pi_1(a)\pi_2(b)=0$. Since $\pi_1(A)^{''}$ is a factor and
$\pi_2(B)^{''} \subseteq \pi_1(A)^{'}$, using \cite[Proposition
IV.4.20]{takesaki}, we see that either $\pi_1(a) =0$ or $\pi_2(b) = 0$,
i.e., $a\otimes b$ belongs to either $A\widehat\otimes F$ or
$E\widehat\otimes B$. In both cases, $a\otimes b \in A\widehat\otimes
F + E\widehat\otimes B$. Finally, exactly on the lines of (1), we
conclude that $K \subseteq A\widehat\otimes F + E\widehat\otimes B$.

(3) If $A$ and $B$ are separable, then every prime ideal is a
primitive ideal. So, the result follows from parts (1) and (2).
\end{proof}
In particular, among all the five proper closed ideals of $B(H)\widehat\otimes B(H)$ - see \cite[Theorem 3.12]{ranj2}- namely, $\{0\},\,B(H) \widehat\otimes K(H),\, K(H) \widehat\otimes B(H), \, B(H) \widehat\otimes K(H)+ K(H) \widehat\otimes B(H)$ and $K(H) \widehat\otimes K(H)$,   the first four are prime as well primitive.

We now discuss the modular ideals of $A\widehat\otimes B$. In a Banach
algebra $A$, an ideal $I$ is said to be {\it modular} (or regular) if there
exists an $e \in A$ such that $xe-x, ex-x \in I$ for all $x\in A$, or
equivalently, if $A/I$ is unital. It is clear that every proper ideal in a
unital Banach algebra is modular. Also, $\{0\}$ is modular if and only
if $A$ is unital.

If $I$ is a closed modular ideal of $A$, then the product ideal $I \widehat\otimes A$ need not be modular in
$A\widehat\otimes A$. This can be seen by considering $A= C_0 (X)$,
where $X$ is a locally compact Hausdorff space (non-compact). A closed
modular ideal of $C_0(X)$ is of the form $I(E)= \{f \in A : f(E)=0\}$,
where $E$ is a compact subset of $X$ \cite{kan}. So let us consider a
closed modular ideal $I= I(E)$ of $A$. Now note that
$$ 
I\widehat\otimes A \subseteq A \widehat\otimes A \subseteq A\otimes_\lambda A = C_0(X \times X),
$$
where `$\lambda$' is the Banach space injective tensor product. This shows that $I
\widehat\otimes A \subseteq I(E\times X) $. Thus, $I\widehat\otimes A$
is not modular, $I(E \times X)$ not being modular. In fact, we have the
following result which characterizes the modular product ideals.

\begin{theorem}
  For closed modular ideals $I$ and $J$ of $A$ and $B$ respectively,
  $I\widehat\otimes J$ is modular in $A\widehat\otimes B$ if and only
  if both $A$ and $B$ are unital.
\end{theorem}

\begin{proof}
  If $A$ and $B$ are both unital, then so is $A\widehat\otimes B$; so
  that every ideal is modular. Conversely, let $I \widehat\otimes J$
  be a modular ideal. Since $ A\widehat\otimes J$ and $ I
  \widehat\otimes B$ both contain $I\widehat\otimes J$, both are
  modular ideals of $A \widehat\otimes B$. Using Lemma \ref{ker}, we have an
 isomorphism between $(A\widehat\otimes B) / (A  \widehat\otimes J) $ and $A \widehat\otimes (B/J)$, and similarly between  $(A\widehat\otimes B) / (I  \widehat\otimes B) $ and $(A/I)\widehat\otimes B$. Therefore, $A \widehat\otimes (B/J)$ and $(A/I)\widehat\otimes B$  are unital which further show that $A$ and $B$ are both unital \cite[Theorem 1]{loy}.
\end{proof}

In particular, $K(H) \widehat\otimes K(H)$ is a closed modular ideal
of $B(H) \widehat\otimes B(H)$, but it is not modular in $B(H)
\widehat\otimes K(H)$. However, the maximal modular ideals behave well in 
$A\widehat\otimes B$ as can be seen in the following result:

\begin{theorem}
  A closed ideal $K$ of $A\widehat\otimes B$ is maximal modular if and
  only if $K=A \widehat\otimes N + M \widehat\otimes B$ for some
  maximal modular ideals $M$ and $N$ of $A$ and $B$, respectively.
\end{theorem}

\begin{proof}
  Let $K$ be a maximal modular ideal of $A\widehat\otimes B$. Since
  every maximal modular ideal is also a maximal ideal, $K$ is of the
  form $K= A\widehat\otimes N + M\widehat\otimes B$ for some maximal
  ideals $M$ and $N$ of $A$ and $B$ respectively \cite[Theorem
  3.11]{ranj2}. Now $(A\widehat\otimes B)/K$ is unital and is
  isomorphic to $A/M\widehat\otimes B/N$, by Lemma \ref{ker} ; therefore, the latter space is unital. But this implies
  that $A/M$ and $B/N$ are both unital \cite[Theorem 1]{loy}. Thus, $M$
  and $N$ are also modular ideals of $A$ and $B$ respectively.

  For the converse, let $K= A\widehat\otimes N + M\widehat\otimes B$,
  where $M$ and $N$ are maximal modular ideals of $A$ and $B$
  respectively. Then, $M$ and $N$ being maximal, by \cite[Theorem
  3.11]{ranj2}, $K$ is also a maximal ideal. Also, the facts that
  $(A\widehat\otimes B)/K$ and $A/M\widehat\otimes B/N$ are
  isomorphic, and $A/M$ and that $B/N$ are both unital, together imply
  that $A\widehat\otimes B/K$ is unital, so that $K$ is modular.
\end{proof}


\section{Wiener Property and Symmetry}
A Banach $\ast$-algebra is said to have {\it Wiener property} if every
proper closed two-sided ideal is annihilated by an irreducible
$\ast$-representation \cite{palmer}. Wiener property for group
algebras and the weighted group algebras has been studied in
\cite{kumar2,ludwig} and others. It is well known that every $C^*$-algebra has
Wiener property.

\begin{theorem}
  The Banach $\ast$-algebra  $A\widehat\otimes B$ has Wiener property.
\end{theorem}

\begin{proof}
Consider a proper closed two-sided ideal $J$ of $A\widehat\otimes B$. Let $J_{\min}$ denote the closure of $i(J)$ in $A\otimes_{\min}B$, where $i:A\widehat\otimes B \rightarrow A\otimes_{\min} B$ is the canonical homomorphism. By \cite[Theorem 6]{kumar}, $J_{\min}$ is also a proper closed two-sided ideal of the $C^*$-algebra $A\otimes_{\min} B$, and so it is annihilated by an irreducible $*$-representation $\pi:A\otimes_{\min} B \rightarrow B(H)$.  Note that the isometry of involution gives $i$ is $*$-preserving, so that we have a $*$-representation $\hat{\pi}:= \pi \circ i$  of $A\widehat\otimes B$ on $H$. Using injectivity of $i$ \cite{jain}, we have $\hat{\pi}(J)=\{0\}$. Also, the relation $\hat{\pi}(A\otimes B) = \pi (A\otimes B)$ gives
\[ \hat{\pi}(A\widehat\otimes B)^{\prime} \subseteq \pi(A\otimes B)^{\prime} = \pi(A\otimes_{\min} B)^{\prime}= \mathbb{C}I,\]
where the equality between the middle expressions follows from the norm density of $\pi(A\otimes B)$ in $\pi(A \otimes_{\min} B)$. This further implies that $\hat{\pi}$ is irreducible; hence, $A\widehat\otimes B$ has Wiener property.
\end{proof}

\vspace*{1mm}
A Banach $*$-algebra is said to be {\it symmetric} if every element of the
form $x^*x$ has positive spectrum, or equivalently, every self adjoint
element has a real spectrum \cite[Theorem 10.4.17]{palmer}. Symmetry in group algebras has been
investigated by various authors, see, for instance,
\cite{ludwig,leptin}. One can easily verify that a Banach
$\ast$-algebra $A$ is symmetric if and only if for every left modular
ideal $I$ of $A$ with modular unit $\alpha$, the set $S_I$ of
Hermitian sesquilinear forms given by
\begin{center}
  $ S_I=\{ B:A\times A \rightarrow \mathbb{C}\, |\, \,B_\alpha=B,
  B(I,A)=\{0\},  B(u,u) \geq 0,$\\
  \hspace*{50mm} $B(uw,vw) = B(v^*uw,w),\, \forall u,v,w \in A\} $
\end{center}
is non-trivial, where $B_\alpha(v,w):= B(v \alpha , w \alpha ),\,
\forall v,w \in A $   \cite{ludwig}. It is well known that every
$C^*$-algebra is symmetric \cite{palmer}. For $C^*$-algebras $A$ and
$B$, we do not know whether the Banach $\ast$-algebra
$A\widehat\otimes B$ is symmetric or not, but if one of them is
subhomogeneous, then we have an affirmative answer. Recall that a
$C^*$-algebra $A$ is {\it subhomogeneous} if there exists a positive integer
$n$ such that each irreducible representation of $A$ has dimension
less than or equal to $n$.

We first modify a result from \cite{kumar} in terms of operator
algebras. We say that a Banach algebra $A$ is an {\it operator algebra} if
there exists a Hilbert space $H$ and a bicontinuous homomorphism of
$A$ into $B(H)$.

\begin{prop}\label{oa}
  If $A$ and $B$ are operator algebras, then $A\widehat\otimes B$ is a
  Banach algebra. If $A$ and $B$ both have isometric involutions then
  $A\widehat\otimes B$ is a Banach $\ast$-algebra.
\end{prop}
\begin{proof}
  It is known that if A is an operator algebra then the multiplication
  operator $m:A\otimes_h A \rightarrow A$ given by $m(a\otimes b)=ab$
  is completely bounded   \cite[Theorem 1.3]{blemer}. Using this
  result, we get the completely bounded operators 
$$
m_A:A\otimes_h A \rightarrow A \ \text{and} \ m_B:B\otimes_h B
\rightarrow B.
$$ 
Now consider the canonical map $i:A\widehat\otimes A \rightarrow
A\otimes_h A$, which is a completely contractive homomorphism. Then, the
multiplication operator $m^{\prime}_{A}:A\widehat\otimes A \rightarrow
A$, which can be regarded as $m^{\prime}_{A} = m_A \circ i$, is completely
bounded. Similarly, the multiplication operator
$m^{\prime}_{B}:B\widehat\otimes B \rightarrow B$ is also completely bounded. In particular, 
the operator 
$$
m^{\prime}_{A}\otimes m^{\prime}_{B} :(A\widehat\otimes
A)\widehat\otimes (B\widehat\otimes B) \rightarrow A\widehat\otimes B
$$ 
is bounded. Using the commutativity of `$\wedge$', the operator
$$
m^{\prime}_{A}\otimes m^{\prime}_{B} :(A\widehat\otimes
B)\widehat\otimes (A\widehat\otimes B) \rightarrow A\widehat\otimes B
$$ 
is also bounded. Hence, $A\widehat\otimes B$ is a Banach algebra. The proof for involution follows as in \cite{kumar}.
\end{proof}

\begin{lemma}\label{fd}
Let  $A$ and $B$ be $C^*$-algebras with either $A$ or $B$ finite-dimensional. Then  $A\widehat\otimes B$ is a symmetric operator algebra.
\end{lemma}

\begin{proof} If $A$ or $B$ is finite dimensional, then clearly, $A\widehat\otimes B$ is $*$-isomorphic to $A\otimes_{\min} B$, which gives the required result.

\end{proof}

\begin{lemma}\label{comm}
  If $A$ is a commutative unital $C^*$-algebra and $B$ is a symmetric
  unital operator algebra with isometric involution, then $A\widehat\otimes B$
  is symmetric.
\end{lemma}

\begin{proof}
  Note that $A\widehat\otimes B$ is a Banach $*$-algebra by Proposition \ref{oa}. Let $\Phi(A)$ denote the set of maximal ideals of $A$, then it is in
  one-one correspondence with the space of non-zero $*$-homomorphisms
  of $A$. For $M \in \Phi(A)$, define $h_M: A\otimes B \rightarrow B$
  by $h_M (\sum a_i \otimes b_i)= \sum a_i(M) b_i$. It is bounded with
  respect to `$\wedge$'-norm, so can be extended to $A\widehat\otimes
  B$ as a $*$-homomorphism. Then, by \cite[Corollary 2]{lebow}, an
  element $x$ of $A\widehat\otimes B$ is invertible if and only if
  $h_M(x)$ is invertible for each maximal ideal $M$ of $A$. Thus,
$$ \sigma(x) = \bigcup_{M \in \Phi(A)} \sigma(h_M(x)), $$
where $\sigma(x)$ denotes the spectrum of $x$ in $A\widehat\otimes
B$. Now consider a self-adjoint element $u$ in $A\widehat\otimes
B$. For any $M \in \Phi(A)$, $h_M$ being $*$-preserving,
$h_M(u)$ is self-adjoint in $B$. But $B$ is symmetric, so
$$ \sigma(u) = \bigcup_{M \in \Phi(A)} \sigma(h_M(u)) \subseteq \mathbb{R}. $$
Hence, $A \widehat\otimes B$ is symmetric.
\end{proof}

\begin{remark}
 Note that one can also prove the above lemma using an argument similar to that in \cite[Corollary 3.3]{bonic}.
\end{remark}

\begin{theorem}\label{symm}
  If $A$ is a subhomogeneous $C^*$-algebra, then for any $C^*$-algebra
  $B$, $A\widehat\otimes B$ is symmetric.
\end{theorem}

\begin{proof}
  Since $A\widehat\otimes B$ can be isometrically embedded in $A^{**}
  \widehat\otimes B^{**}$ as a closed $*$-subalgebra, it is sufficient
  to show that $A^{**} \widehat\otimes B^{**}$ is symmetric. Let $A$
  be $n$-subhomogeneous, then $A^{**}$ is a direct sum of type $I_m$
  von Neumann algebras for $m \leq n$   \cite[Theorem
  IV.1.4.6]{blackadar}. Also each type $I_m$ von Neumann algebra is
  isomorphic to $M_m \overline{\otimes} C$,
  where $M_m$ is the set of $m \times m$ complex matrices and $C$ is a
  commutative von Neumann algebra \cite[III.1.5.12]{blackadar}. Thus,
  $A^{**} \widehat\otimes B^{**}$ is $*$-isomorphic (not necessarily
  isometrically) to a direct sum of some $M_m(C) \widehat\otimes
  B^{**}$. For each $m$, $M_m(C)$ is isomorphic to $M_m
  \widehat\otimes C$; so, using the commutativity and associativity of
  the operator space projective norm, we get $M_m(C) \widehat\otimes
  B^{**}$ is $*$-isomorphic to $C \widehat\otimes (M_m \widehat\otimes
  B^{**})$. Note that, Lemma \ref{fd} gives  $M_m
  \widehat\otimes B^{**}$ is an operator algebra with an isometric involution
  and is symmetric; so, by Lemma \ref{comm}, $M_m(C) \widehat\otimes
  B^{**}$ is symmetric. Hence, $A^{**} \widehat\otimes B^{**}$ is
  symmetric being the direct sum of symmetric Banach $*$-algebras \cite[Theorem 11.4.2]{palmer}
\end{proof}

\begin{remark}
  If $A$ is commutative and $B$ is any $C^*$-algebra, then, by
  \cite[Corollary 3.3]{bonic}, $A\otimes_\gamma B$ is
  symmetric. However, the symmetry of $A\otimes_\gamma B$ when $A$ is
  subhomogeneous and $B$ is any $C^*$-algebra follows as in Theorem
  \ref{symm}.
\end{remark}


\end{document}